\documentclass[a4paper,fleqn]{article}
\usepackage{amsmath}
\usepackage{latexsym}
\usepackage{amssymb}
\usepackage{amsfonts}
\usepackage{rsfs}
\usepackage[dvips]{graphics}
\usepackage[all]{xy}
\usepackage{amsthm}
\usepackage{enumerate}
\usepackage{url}

\DeclareMathOperator{\Fin}{\mathit{Fin}}

\newcommand{\mon}[2]{\mathrm{mon}_{\mathscr{#1}}\left(#2\right)}

\newcommand{\Fun}[2]{\mathit{Fun}\left({#1},{#2}\right)}
\newcommand{\F}[2]{\mathscr{F}\left({#1},{#2}\right)}
\newcommand{\FC}[2]{\mathscr{F}_C\left({#1},{#2}\right)}
\newcommand*{\FN}{\mathit{FN}}
\newcommand*{\FQ}{\mathit{FQ}}
\newcommand*{\BQ}{\mathit{BQ}}

\newcommand*{\id}{\mathit{id}}
\let\uhr\upharpoonright
\renewcommand*{\upharpoonright}{\hspace{-.07cm}\uhr\hspace{-.07cm}}
\newtheorem{thm}{Theorem}
\newtheorem{prop}[thm]{Proposition}

\theoremstyle{remark}

\newtheorem{ex}{Example}
\theoremstyle{definition}

\newtheoremstyle{axiom1}
{3pt}
{3pt}
{\itshape}
{}
{\bfseries\itshape}
{}
{.5em}
{}
\newtheoremstyle{axiom2}
{3pt}
{3pt}
{\itshape}
{}
{\bfseries\itshape}
{\\}
{.5em}
{}
\theoremstyle{axiom1}
\title{Real Functions and its Differentiation\\ in Alternative Set Theory}
\author{Kiri Sakahara\thanks{Yokohama National University and Kanagawa University, Kanagawa, Japan.} \and Takashi Sato\thanks{Toyo University, Tokyo, Japan.}}
\date{}
\begin{document}

\maketitle
\begin{abstract}
In the previous paper (Kiri Sakahara and Takashi Sato. Basic Topological Concepts and a Construction of Real Numbers in Alternative Set Theory. arXiv e-prints, arXiv:2005.04388, May 2020), the authors displayed basic topological concepts and a construction of a system of real number in alternative set theory (AST).
The present paper is a continuation of that research providing additional treatments of real functions.
The basic properties of differentiation in AST are preserved as in the conventional calculus.
\end{abstract}

\section{Introduction}
The authors displayed basic topological concepts and a construction of a system of real number in alternative set theory (AST, for short) in Sakahara and Sato \cite{topology}.
The present paper is a continuation of that research providing additional treatments of real functions.

Almost all concepts, formulations and statements displayed here are due to Tsujishita \cite{tjst}.
The only exception is Theorem \ref{6}, which asserts that derivatives of real functions given here coincide with that defined in the traditional manner.
Proofs of the statements are also virtually identical.
However, the paper does not skip many of them since there remain essential differences with regard to the way sets are formulated, so too the way proofs proceed.

The readers can find relatively compact explanation of AST, such as an axiomatic system of AST, in Vop\v{e}nka and Trlifajov\'{a} \cite{encycro-ast} and Sakahara and Sato \cite{cogjump,topology}.

\section{A system of numbers}

Let us start with constructing a number system in accordance with Vop\v{e}nka \cite{ast}.
The class of \textit{natural numbers} $N$ is defined as:
\[
N\ =\ \left\{x\ ;\
\begin{matrix}
\left(\forall y\in x\right) \left(y\subseteq x\right)\\
\wedge\left(\forall y,z\in x \right) \left(y\in z \vee y=z \vee z\in y\right)
\end{matrix}
\right\},
\]
while the class of \textit{finite natural numbers} $\FN$ consists of the numbers represented by finite sets
\[
\FN \ =\ \left\{x\in N\ ;\
\Fin(x)\right\}
\]
in which $\Fin(x)$ means that each subclass of $x$ is a set.
The class of all integers $Z$ and that of all rational numbers are defined respectively as:
\[
Z\ =\ N\cup \left\{ -a\ ;\ a\in N \wedge a\ne 0\right\}
\qquad\text{and}\qquad
Q\ =\ \left\{
\frac{x}{y}\ ;\ x,\,y\in Z \wedge y\ne 0
\right\}.
\]
$\BQ\subseteq Q$ denotes the class of \textit{bounded rational numbers} and $\FQ\subseteq BQ$ the class of \textit{finite rational numbers}, i.e.,
\begin{eqnarray*}
\BQ &=& \{q\in Q\ ;\ \left(\exists i\in\FN\right)
  \left(|q|\leq i\right)
\}\quad\text{ and }\\
\FQ &=& \left\{ q \in Q\ ;\ \left(\exists x,y\in\FN\right)
\left( q=\frac{x}{y}\vee q=-\frac{x}{y}\right) \right\}.
\end{eqnarray*}

Real numbers are defined in AST as an equivalence class of bounded rational numbers.
The reason behind this construction lies in the human's inability to distinguish two mutually close rational numbers.
This idea is grasped by the \textit{indiscernibility equivalence}, $\doteq$, on the class $Q$ of all rational numbers.
One of the definitions of the indiscernibility equivalence $\doteq$ is given as:
\[
p\doteq q\quad \equiv\quad
 \begin{pmatrix}
  \left( \exists k \right)
  \left( \forall i>0\right)
  \left(|p|<k \wedge |p-q|<\frac{1}{i}\right)\\[.18cm]
  \vee
  \left(\forall k\right)
  \left(
  \left( p>k \wedge q>k\right)  \vee
  \left(p<-k \wedge q<-k\right)
  \right)
 \end{pmatrix}
\]
in which the letters $i,j,k$ denote finite natural numbers, i.e., $i,j,k\in\FN$ for notational ease, hereafter.
For each $q\in Q$ the notation $\mon{}{q}=\{s\in Q\, ;\, s\doteq q\}$, called the \textit{monad} of $q$, represents the class of all rational numbers which are indiscernible from $q$.
A real number is denoted as a monad $\mon{}{q}$ of some rational number $q$.
Two limiting cases are denoted as:
\[
\infty\ =\
\{q\in Q\ ;\
 \left(\forall i\right)
 \left(q>i\right)
\}\quad\text{ and }\quad
-\infty\ =\
\{q\in Q\ ;\
 \left(\forall i\right)
 \left(q<-i\right)
\}.
\]
The class of all real numbers $R$ is defined as:
\[
R\ \equiv\ \left\{\mon{}{x}\ ;\
x\in \BQ
\right\}\ =\ \BQ/\doteq.
\]
Let us denote a \textit{real continuum} as $\mathscr{R}=\langle Q,\doteq\rangle$, where a \textit{continuum}\footnote{This concept is due to Tsujishita \cite{tjst}.} is a pair of classes $\mathscr{C}=\left\langle C,\doteq_C\right\rangle$, in which a set-theoretically definable class $C$ is called as a support of $\mathscr{C}$.

Finite arithmetic operation of real numbers is same as usual.
The countable sum of nonnegative real number $r_i$ is defined by the following.
Let $b_i\in r_i$.
Since the numbers $r_i$ are nonnegative, prolonging the sequence $(b_i)_{i\in\FN}$ onto a set $(b_i)_{i\in \alpha}$, there exists an infinite natural number $\delta\leq \alpha,\, \delta\notin\FN$ which satisfies for any infinite natural number $\gamma\leq \delta$, $\gamma\notin\FN$ the following indiscernibility equivalence
\[
  \sum_{n=0}^\gamma b_n\ \doteq\ \sum_{n=0}^\delta b_n.
\]
We put $\sum_{i\in\FN} r_i=r$, where $r\in {R}\cup\{\pm\infty\}$ satisfies $\sum_{n=0}^\delta b_n\in r$.

\section{Morphisms}

Let $\mathscr{C}_i\ (i=1,2)$ be continua, then a function $F:C_1\rightarrow C_2$ is \textit{continuous} if $x\doteq y$ implies $F(x)\doteq F(y)$.
Two continuous functions $F$ and $G$ are \textit{indiscernible}, denoted simply as $F\doteq G$, if for all $x\in C_1$, $F(x)\doteq G(x)$ follows\footnote{
  To be precise,  put
  \[
  r_i\ \equiv\
  \left\{\langle f,g\rangle\ ;\
  \left(f,g\in \Fun{c_1}{c_2}\right)
  \wedge\left(\forall x\in c_1\right)\left(f(x)-g(x)\leq\frac{1}{2^i}\right)
  \right\},
  \]
  in which $C_j\subseteq c_j$ for $j=1,2$ when $C_1$ is a semiset, otherwise $H(C_j)\subseteq c_j$ for a given similarity endomorphism $H:V\rightarrow D$ in which $D$ is a semiset (for the definition and its existence see p.111 of Vop\v{e}nka \cite{ast}), and $\doteq_{c_2^{c_1}}\ \equiv\ \bigcap_{i\in\FN} {r}_i$.

  For each pair of functions $F,G\in\Fun{C_1}{C_2}$,
  let us denote $F\doteq G$ iff there exists their (or their similar classes') prolonged sets of functions $f,g\in\Fun{c_1}{c_2}$ which satisfies $F=f\upharpoonright C_1$ (or $H(F)=f\upharpoonright H(C_1)$), $G=f\upharpoonright C_1$ (or $H(G)=g\upharpoonright H(C_2)$) and $f\doteq_{c_2^{c_1}} g$.
}.

The \textit{morphism} $\mathscr{F}$ between two continua is defined as follows.
\begin{enumerate}[(1)]
\item A \textit{morphism} from $\mathscr{C}_1$ to $\mathscr{C}_2$ is a monad $\mon{}{F}$ denoted simply as $[F]$ for some continuous function $F$ from $C_1$ to $C_2$.
\item If $\mathscr{C}_i$ ($i=1,2$) are continua, the notation  $\mathscr{F}:\mathscr{C}_1\rightarrow \mathscr{C}_2$ means that $\mathscr{F}$ is a morphism from $\mathscr{C}_1$ to $\mathscr{C}_2$ .
\item If $\mathscr{F}$ is a morphism, then the expression $G\in \mathscr{F}$ means $G\doteq F$ in which $\mathscr{F}=[F]$.
  If $G\in \mathscr{F}$, we say that the morphism $\mathscr{F}$ is \textit{represented by} $G$ and $G$ \textit{represents} $\mathscr{F}$.
\item If $\mathscr{F}$ and $\mathscr{G}$ are morphisms represented respectively by $F$ and $G$, then the expression $\mathscr{F}=\mathscr{G}$ means $F\doteq G$.
\end{enumerate}
It is essential that morphisms are defined as the monads of continuous function.
When a set-theoretically definable function $F:C_1\rightarrow C_2$ has an indiscernible gap at $x\in C_1$, that is, $\neg(F(x)\doteq F(y))$ for some $y\doteq x$, its value at $\mon{}{x}$ cannot be determined uniquely since $x\doteq_{C_1} y$ but $\neg(F(y)\doteq_{C_2} F(x))$, thus, $\mon{C_2}{F(x)}\cap\mon{C_2}{F(y)}=\emptyset$.

Contrary to the framework of Tsujishita, in which the morphisms are not guaranteed to be classes, they are in AST.

The identity morphism $\id_\mathscr{C}$ is represented by the identity function $\id_C$.
Given $\mathscr{F}_1:\mathscr{C}_1\rightarrow \mathscr{C}_2$ and $\mathscr{F}_2:\mathscr{C}_2\rightarrow\mathscr{C}_3$, the composition $\mathscr{F}_2\circ \mathscr{F}_1:\mathscr{C}_1\rightarrow\mathscr{C}_3$ is given as the morphism $[F_2\circ F_1]$ represented by the composition $F_2\circ F_1$ of $F_1:C_1\rightarrow C_2$ and $F_2:C_2\rightarrow C_3$.

A morphism $\mathscr{F}:\mathscr{C}_1\rightarrow\mathscr{C}_2$ is \textit{injective} and \textit{surjective} if it is represented by a continuous function $F:C_1\rightarrow C_2$ satisfying respectively
\[
\left(\forall x,y\in C_1\right)
\left(F(x)\doteq F(y) \text{ implies } x\doteq y\right).
\]
and
\[
\left(\forall x_2\in C_2\right)
\left(\exists x_1\in C_1\right)
\left(F(x_1)\doteq x_2\right)
\]

A morphism $\mathscr{F}:\mathscr{C}_1\rightarrow \mathscr{C}_2$ is an \textit{equivalence} if there is a uniquely determined morphism $\mathscr{F}^{-1}$, which is the \textit{inverse} of $\mathscr{F}$, satisfying \[
\mathscr{F}^{-1}\circ \mathscr{F}=\id_{\mathscr{C}_1}\ \text{ and }\ \mathscr{F}\circ \mathscr{F}^{-1}=\id_{\mathscr{C}_2}.
\]
Every pair of representations of $\mathscr{F}$ and $\mathscr{F}^{-1}$, say $F:C_1\rightarrow C_2$ and $F^{-1}:C_2\rightarrow C_1$, satisfies
\[
F^{-1}\circ F\ \doteq \ \id_{C_1}\quad\text{and}\quad
F\circ F^{-1}\ \doteq \ \id_{C_2}.
\]
$F^{-1}$ is said to be \textit{almost inverse} of $F$.

If there is an equivalence $\mathscr{F}:\mathscr{C}_1\rightarrow\mathscr{C}_2$, the continuum $\mathscr{C}_1$ is said to be \textit{equivalent} to $\mathscr{C}_2$ denoted as $\mathscr{C}\simeq\mathscr{C}_2$.

Let $\mathscr{C}$ be a continuum and $\mathscr{C}_i\subset \mathscr{C}$ ($i=1,2$) be subcontinua.
An equivalence $\mathscr{A}:\mathscr{C}_1\rightarrow\mathscr{C}_2$ is said to be a \textit{quasi-identity} if its representation satisfies $\alpha(x)=x$ for all $x\in C_1$.
A quasi-identity, by definition, is uniquely determined if it exists.

\begin{prop}[Proposition 2.4.2 of Tsujishita \cite{tjst}]\label{242}
  Let $r$ be a nonzero nonfinite rational number, in which there exists $\tau\in N\setminus\FN$ and $r=\frac{1}{\tau}$.
  The inclusion
  \[
  \iota_r:r Z\rightarrow Q
  \]
  represents a quasi-identity, for which the function $\kappa_r:Q\rightarrow rZ$ defined by
  \[
  \kappa_r(s)\ \equiv\ \left[\frac{s}{r}\right]r
  \]
  in which $[x]$ denotes the integer part of $x$, gives an almost inverse of $\iota_r$.
\end{prop}
\begin{proof}
  It is obvious that $\kappa(\iota(nr))=nr$ for $n\in N$.
  Since $x<[x+1]\leq x+1$, the following inequations follow
  \[
  s\ =\ \left(\frac{s}{r}\right)r\ <\ \left[\frac{s}{r}+1\right] r\ \leq\ \left(\frac{s}{r}+1\right)r\ =\ s+r,
  \]
  thus, $\left[\frac{s}{r}\right]r\doteq s$ follows since $r\doteq 0$.
  Finally, $\iota\circ\kappa\doteq id$, follows.
\end{proof}

\section{Continua of functions}
Let $\mathscr{C}_i$ $(i=1,2)$ be continua in which $C_1$ is set-theoretically definable\footnote{
A class $\{ x\, ;\, \varphi(x)\}$ is said to be \textit{set-theoretically definable} if $\varphi(x)$ is a set formula.
}.
The \textit{continuum of functions} are given as
\[
\F{\mathscr{C}_1}{\mathscr{C}_2}\ =\ \left\langle\Fun{C_1}{C_2},\,\doteq\right\rangle.\footnotemark
\]
\footnotetext[4]{
Since $C_1$ is not always a set, $\Fun{-}{-}$, which is a class of all functions from left$-$ to right$-$, remains a codable class.
In regard to codable class, see 1.5 of Vop\v{e}nka \cite{ast}.
}
Subcontinua of $\F{\mathscr{C}_1}{\mathscr{C}_2}$ consisting of continuous functions are \textit{continua of morphisms from} $\mathscr{C}_1$ to $\mathscr{C}_2$, denoted as $\FC{\mathscr{C}_1}{\mathscr{C}_2}$.

While $\F{\mathscr{C}_1}{\mathscr{C}_2}$ is not necessarily equivalent to $\F{\mathscr{C}'_1}{\mathscr{C}'_2}$ even if $\mathscr{C}_1\simeq \mathscr{C}'_1$, continua of morphism $\FC{\mathscr{C}_1}{\mathscr{C}_2}$ always do as it is shown in the next proposition.

\begin{prop}[Proposition 5.1.1 of Tsujishita \cite{tjst}]\label{511}
  Let $\mathscr{C}_i$ and $\mathscr{C}'_i$ $(i=1,2)$ be continua and $G_i:C_i\rightarrow C'_i$ be representations of equivalences with almost inverse $G^{-1}_i:C'_i\rightarrow C_i$.
  Define functions from $\Fun{C_1}{C_2}$ to $\Fun{C'_1}{C'_2}$ as
  \[
  \alpha(F)\ \equiv\ G_2\circ F\circ G^{-1}_1,\hspace{\fill}
  \vcenter{
    \xymatrix{
      C_1
      \ar[r]^{F}
      \ar[d]^{\simeq}_{G_1}
      &
      C_2
      \ar[d]_{\simeq}^{G_2}
      \\
      C'_1
      \ar[r]_{\alpha(F)}
      &
      C'_2
    }
  }\hspace{1.44cm}
  \]
  and from $\Fun{C'_1}{C'_2}$ to $\Fun{C_1}{C_2}$ as
  \[
  \beta(F')\ \equiv\ G^{-1}_2\circ F'\circ G_1. \hspace{\fill}
  \vcenter{
    \xymatrix{
      C_1
      \ar[r]^{\beta(F')}
      \ar[d]^{\simeq}_{G_1}
      &
      C_2
      \ar[d]_{\simeq}^{G_2}
      \\
      C'_1
      \ar[r]_{F'}
      &
      C'_2
    }
  }\hspace{1.44cm}
  \]
  Then, $\alpha$ represents an equivalence $\mathscr{A}:\mathscr{F}_C(\mathscr{C}_1,\mathscr{C}_2)\rightarrow\mathscr{F}_C(\mathscr{C}'_1,\mathscr{C}'_2)$ with an almost inverse $\mathscr{B}$, one of whose representations is $\beta$.
\end{prop}
\begin{proof}
  Let $F,H\in\mathscr{F}_C(\mathscr{C}_1,\mathscr{C}_2)$ which satisfies $F\doteq H$ and $x\in C'_1$.
  Then $F(G^{-1}_1(x)) \doteq H(G^{-1}_1(x))$ holds.
  Consequently the next equations hold
  \[
  \alpha(F)(x)\ =\ G_2(F(G_1^{-1}(x)))\ \doteq\ G_2(H(G_1^{-1}(x)))\ =\ \alpha(H)(x)
  \]
  and, thus, $\alpha$ is continuous.
  Continuity of $\beta$ is verified in a similar way.

  Next, we show that $\beta$ is an almost inverse of $\alpha$.
  Since $G^{-1}_i$ are almost inverse of $G_i$, $G^{-1}_i(G_i(x))\doteq x$ holds for every $x\in C_i$ and $i=1,2$.
  Then $F(G^{-1}_i(G_i(x)))\doteq F(x)$ follows by the continuity of $F$.
  Consequently,
  \[
  \beta(\alpha(F))(x)\ =\ G_2^{-1}(G_2(F(G_1^{-1}(G_1(x)))))\ \doteq\ F(x)
  \]
  and the equation below holds
  \[
  \beta(\alpha(F))\ \doteq\ F.
  \]
  Similarly, $\alpha(\beta(F'))\doteq F'$ follows.
\end{proof}

\section{Real functions}
Let $\mathscr{C}$ be a continuum.
A \textit{real function on} $\mathscr{C}$ is a morphism from $\mathscr{C}$ to $\mathscr{R}$.

\begin{ex}[Polynomial functions]
  The polynomial function $\mathscr{F}:\mathscr{R}\rightarrow\mathscr{R}$ is given simply as
  \[
  \mathscr{F}(x)\ \equiv\ \sum_{i=0}^{n}p_i\cdot x^i.
  \]
  in which $p_i\in R$.
\end{ex}

\begin{ex}[Exponential]
  For $\tau\in N\setminus\FN$ and $q\in Q$, define the approximation of the $\tau$th power of the Napier's constant by
  \[
  \exp(q,\tau)\ \equiv\ \sum_{i=0}^{\tau}\frac{q^i}{i!}.
  \]
  Uniqueness of the value of the exponential is verified by Proposition 7.2.1 of Tsujishita \cite{tjst}.
  Then the \textit{exponential function} is given as
  \[
  e^x\ =\ \exp(x)\ \equiv\ \mon{}{\exp(q,\tau)}\quad\text{where } q\in x.
  \]
  in which the choices of $\tau\in N\setminus\FN$ and $q\in x$ are arbitrary.
  Uniqueness of its value of the function also verified by Proposition 7.2.5 and fundamental properties are verified by Propositions 7.2.6 and 7.2.10 of Tsujishita \cite{tjst}.

  It is worth mentioning that one of the usual ways to define exponential function is given as:
  \[
  \exp(x)\ \equiv\ \sum_{t=0}^{\infty}\frac{x^t}{t!}.
  \]
  As it is defined in the last paragraph of Section 2, countable sum of a given sequence of real numbers $\left(\frac{x^t}{t!}\right)_{t\in\FN}$ is given by that of prolonged sequence of their representative rational numbers, say, $\left(\frac{q^t}{t!}\right)_{t\in\tau}$ in which $q\in x$.
\end{ex}

\begin{ex}[Logarithm]
  For $\tau\in N\setminus\FN$ and $q\in\{x\in Q\, ;\, x>0\}$, define the approximation of the natural logarithm of $q$ by
  \[
    \log(q,\tau)\ \equiv\ \frac{1}{\tau}\max\left\{k\in Z\ ;\ \left(\exp\left(\frac{k}{\tau},\tau\right)\leq q\right)\wedge \left(|k|\leq \tau^2\right)\right\}.
  \]
  Then the \textit{natural logarithm function} can be given as
  \[
  \log(x)\ \equiv\ \mon{}{\log(q,\tau)}\quad \text{ where }q\in x,\text{ for } x\in(0,\infty).
  \]
  in which the choices of $\tau\in N\setminus\FN$ and $q\in x$ are arbitrary.
  Uniqueness of its value and fundamental properties are verified in Proposition 7.2.12 and 7.2.13 of Tsujishita \cite{tjst}.

  As it is seen in the last example, the definition of the function coincides with one of the usual ways to define it.
  Not only these two cases, every traditional real function can be dealt exactly the same manner.
\end{ex}

In the remainder of the section, let us investigate behaviors of real functions divided by indiscernible intervals to probe into features regarding derivatives of real functions, which we take up in the next section.

\begin{prop}[Proposition 7.5.1 of Tsujishita \cite{tjst}]\label{751}
  Let $F,\, G$ be rational-valued functions on $C$ and $a,b\in C$ which are mutually indiscernible, that is, $F\doteq G$ and $a\doteq_C b$, then the following two conditions are equivalent: for any given finite $n\in\FN$
  \begin{equation}\label{29}
    \left(
    \left(
    \left(\exists \varepsilon_a\in\mon{}{0}\right)
    \left(d(x,a)>|\varepsilon_a|\right)\wedge \left(x\doteq_C a\right)
    \right)\Rightarrow
    \left(\frac{|F(x)|}{d(x,a)^n}\doteq 0\right)\right)
  \end{equation}
  and
  \begin{equation}\label{30}
  \left(
  \left(
  \left(\exists \varepsilon_b\in\mon{}{0}\right)
  \left(d(x,b)>|\varepsilon_b|\right)\wedge \left(x\doteq_C b\right)
  \right)\Rightarrow
  \left(\frac{|G(x)|}{d(x,b)^n}\doteq 0\right)\right).
  \end{equation}
\end{prop}
\begin{proof}
  Suppose that the property (\ref{29}) holds.
  Choose $x\in C$ which satisfies $d(x,a)>|\varepsilon_a|$ and $x\doteq_C a$.
  Put $\varepsilon_b$ as:
  \[
  \varepsilon_b\ = \ \max{\left\{|\varepsilon_a|+d(a,b),\,3\cdot d(a,b)\right\}}\ \doteq \ 0.
  \]
  Then for all $x\in C$ which satisfies $0\doteq d(x,b)>\varepsilon_b$, the following inequations follow by the triangle inequality:
  \[
  d(x,a)\ \geq\ d(x,b)-d(a,b)\ > \ |\varepsilon_a|,
  \]
  and
  \[
  d(x,a) \ \geq\ d(x,b)-d(a,b)\ >\ 3\cdot d(a,b)-d(a,b)\ =\ 2\cdot d(a,b).
  \]
  The last inequality implies that
  \[
  d(x,b)\ \geq\ d(x,a)-d(a,b)\ >\ \frac{1}{2}\cdot d(x,a),
  \]
  and, therefore, the indiscernibility equivalence below follows:
  \[
  \frac{|G(x)|}{d(x,b)^n}\ \doteq\ 2^n\cdot\frac{|F(x)|}{d(x,a)^n}\ \doteq 0.
  \]
  The converse is true by the similar argument.
\end{proof}

\begin{prop}[Proposition 7.5.3 of Tsujishita \cite{tjst}]\label{753}
  Let $\varepsilon,\eta\in\left\{\frac{1}{\gamma}\, ;\, \gamma\in N\setminus\FN\right\}$, $\alpha:[0,1]^n_{\varepsilon}\rightarrow[0,1]^n_\eta$, where $[0,1]_\varepsilon \equiv \left\{x\varepsilon\, ;\, x\in\left[ 0,\frac{1}{\varepsilon}\right]\right\}$, represents an equivalence, and $F$ be a continuous rational-valued function on $[0,1]^n_{\eta}$.
  Then for $n\in\FN$ and $a\in[0,1]^n_\varepsilon$,
  \begin{equation}\label{34}
    \left(
    \left(
    \left(\exists \delta_\eta\in\mon{}{0}\right)
    \left(d(y,\alpha(a))>|\delta_\eta|\right)\wedge \left(y\doteq \alpha(a)\right)
    \right)\Rightarrow
    \left(\frac{|F(y)|}{d(y,\alpha(a))^n}\doteq 0\right)\right)
  \end{equation}
  if and only if
  \begin{equation}\label{35}
  \left(
  \left(
  \left(\exists \delta_\varepsilon\in\mon{}{0}\right)
  \left(d(x,a)>|\delta_\varepsilon|\right)\wedge \left(x\doteq a\right)
  \right)\Rightarrow
  \left(\frac{|F(\alpha(x))|}{d(x,a)^n}\doteq 0\right)\right).
  \end{equation}
\end{prop}

\begin{proof}
  Suppose the property (\ref{34}) holds.
  Let $\beta$ be an almost inverse of $\alpha$.
  Since $\alpha$ represents an equivalence there exists $\delta\in\mon{}{0}$ which satisfies
  \[
  d(x,\alpha(x))<|\delta|
  \]
  for all $x\in[0,1]_\varepsilon$.
  By the triangle inequality, the following inequalities are fulfilled for all $a,x\in[0,1]_\varepsilon$
  \[
  d(\alpha(x),\alpha(a))
  \ \leq\ d(x,\alpha(x))+d(x,a)+d(a,\alpha(a))
  \ <\ d(x,a)+2|\delta|.
  \]
  Let us choose $x\in [0,1]_\varepsilon$ which satisfies $d(x,a)>4|\delta|$ and $x\doteq a$.
  Then, the following inequality is drawn from the last one.
  \[
  d(\alpha(x),\alpha(a))\ <\ \frac{3}{2} d(x,a).
  \]
  It implies that the following indiscernible equality.
  \[
  \frac{|F(\alpha(x))|}{d(x,a)^n}
  \ \leq\
  \left(\frac{3}{2}\right)^n \cdot  \frac{|F(\alpha(x))|}{d(\alpha(x),\alpha(a))^n}
  \ \doteq\ 0.
  \]

  The converse case can also be shown by almost the same argument.
  Suppose the property (\ref{35}) holds.
  Put $b=\alpha(a)$.
  Since $a\doteq \beta(b)$, the next property is met
  \[
  \left(
  \left(
  \left(\exists \delta_b\in\mon{}{0}\right)
  \left(d(x,\beta(b))>|\delta_b|\right)\wedge \left(x\doteq \beta(b)\right)
  \right)\Rightarrow
  \left(\frac{|F(\alpha(x))|}{d(x,\beta(b))^n}\doteq 0\right)\right).
  \]
  The following inequality is also satisfied since $\beta$ is an almost inverse of $\alpha$
  \[
  d(y,\beta(y))\ <\ |\gamma|
  \]
  for all $y\in[0,1]_{\eta}$.
  By the same argument as the previous case, the inequality below follows:
  \[
  d(\beta(y),\beta(b))
  \ \leq\ d(y,\beta(y))+d(x,b)+d(b,\beta(b))
  \ <\ d(y,b)+2|\gamma|.
  \]
  Let us choose $y\in [0,1]_\eta$ which satisfies $d(y,b)>4|\gamma|$ and $y\doteq b$.
  Then, the following inequality is drawn from the last one.
  \[
  d(\beta(y),\beta(b))\ <\ \frac{3}{2} d(y,b)\ =\ \frac{3}{2} d(y,\alpha(a)) .
  \]
  It implies that the following inequality:
  \[
  \frac{|F(y)|}{d(y,\alpha(a))^n}
  \ \doteq\
  \frac{|F(\alpha(\beta(y)))|}{d(y,\alpha(a))^n}
  \ \leq\
  \left(\frac{3}{2}\right)^n \cdot  \frac{|F(\alpha(\beta(y)))|}{d(\beta(y),\beta(b))^n}
  \ \doteq\ 0.
  \]
\end{proof}

\section{Differentiation}

Let $\varepsilon>0$ be a nonfinite rational number $\varepsilon=\frac{1}{\tau}$ for $\tau\in N\setminus \FN$, and $([0,1]_{\varepsilon},\doteq)$ be a continuum representing a real interval $[0,1]$, in which $[0,1]_\varepsilon\equiv\left\{\frac{n}{\tau}\, ;\, (0\leq n\leq \tau)\wedge(n\in N)\right\}$.

Let $(f,[0,1]_{\varepsilon})$ denote a pair of a continuous rational-valued function $f$ on $[0,1]_\varepsilon$ and a $[0,1]_\varepsilon$-valued function $\kappa_\varepsilon:[0,1]\rightarrow [0,1]_\varepsilon$ given in Proposition \ref{242} as
\[
\kappa_\varepsilon (s)\ =\ \left[\frac{s}{\varepsilon}\right]\varepsilon.
\]
It is said that a real function $\mathscr{F}$ is \textit{represented by} $(f,[0,1]_{\varepsilon})$, or simply \textit{by} $f$\footnote{The existence of such $f$ is guaranteed by Proposition \ref{242} and \ref{511}.} iff $f\circ \kappa_\varepsilon$ represents $\mathscr{F}$.

Let us also denote $[0,1]_\varepsilon\setminus \{1\}$ as $[0,1]_\varepsilon^-$, and $x^+\equiv x+\Delta x$ in which $\Delta x\equiv\varepsilon$.
Then a \textit{difference quotient} of a rational-valued function $f$ on $[0,1]_\varepsilon$ at $x\in[0,1]_\varepsilon$ is given as
\[
\frac{\Delta f}{\Delta x}(x)\ \equiv\ \frac{f(x^+)-f(x)}{\Delta x}.
\]
$\Delta f(x)\ \equiv\ f(x^+)-f(x)$ is a \textit{difference} of $f$ at $x$, and $\frac{\Delta f}{\Delta x}$ is a \textit{difference quotient function} of $f$.

At a first glance, it may seem that the difference quotient function depends on the choice of $\varepsilon$, but it is not.
It is verified by the next proposition.

\begin{prop}[Proposition 8.2.1 of Tsujishita\cite{tjst}]\label{821}
  Let $\mathscr{F}$ be a real function on $[0,1]$.
  Let $\left(f_i,[0,1]_{\varepsilon_i}\right)$ $(i=1,2)$ be representations of $\mathscr{F}$ such that the difference quotients of $f_i$ are continuous.
  Then the real functions on $[0,1]$ represented by the difference quotients $\left(\frac{\Delta f_i}{\Delta x}, [0,1]_{\varepsilon_i}^-\right)$ $(i=1,2)$ coincide.
\end{prop}

To prove the proposition, let us first confirm that difference quotients coincide with derivatives.
\begin{thm}[Theorem 8.3.1 of Tsujishita \cite{tjst}]\label{831}
  If $f$ is a function on $[0,1]_\varepsilon$ with continuous difference quotients, then for $a\in[0,1]_\varepsilon^-$,
  \[
  \left(
  \left(
  \left(\exists\eta\in\mon{}{0}\right)
  \left(x-a>|\eta|\right)\wedge
  \left(x\doteq a\right)
  \right)\Rightarrow
  \left(\frac{f(x)-f(a)}{x-a}-\frac{\Delta f}{\Delta x}(a)\doteq 0\right)
  \right).
  \]
\end{thm}
\begin{proof}
  Pick $x\in[0,1]_\varepsilon$ which satisfies $x\doteq a$ and $a<x$, then
  \begin{eqnarray*}
    f(x)-f(a) & = & \sum_{a\leq u<x}\frac{\Delta f}{\Delta x}(u)(u^+-u)\\
    & = & \sum_{a\leq u<x}\frac{\Delta f}{\Delta x}(a)(u^+-u) +  \sum_{a\leq u<x}
    \left(\frac{\Delta f}{\Delta x}(u) - \frac{\Delta f}{\Delta x}(a)\right)(u^+-u)\\
    & = & \frac{\Delta f}{\Delta x}(a)(x-a) +  \sum_{a\leq u<x}
    \left(\frac{\Delta f}{\Delta x}(u) - \frac{\Delta f}{\Delta x}(a)\right)(u^+-u).
  \end{eqnarray*}
  Then, the following equation is met:
  \[
  f(x)-f(a)-\frac{\Delta f}{\Delta x}(a)(x-a)
  \ =\
  \sum_{a\leq u<x}\left(\frac{\Delta f}{\Delta x}(u) - \frac{\Delta f}{\Delta x}(a)\right)(u^+-u).
  \]
  Since $\frac{\Delta f}{\Delta x}$ is continuous and $a\doteq x$, $\left|\frac{\Delta f}{\Delta x}(u) - \frac{\Delta f}{\Delta x}(a)\right|<|q|$ for all $q\in\FQ$.
  Thus, for any $q\in\FQ$, the next inequations are met.
  \[
  \left|
  \sum_{a\leq u<x}\left(\frac{\Delta f}{\Delta x}(u) - \frac{\Delta f}{\Delta x}(a)\right)(u^+-u)
  \right|
  \ <\ \frac{x-a}{\varepsilon}\cdot \left|q\right|\cdot \varepsilon\ =\ (x-a)\cdot \left|q\right|.
  \]
  It implies $\left|f(x)-f(a)-\frac{\Delta f}{\Delta x}(a)(x-a)\right|<(x-a)\cdot |q|$ for all $q\in\FQ$, and thus,
  \[
  \frac{f(x)-f(a)}{x-a}-\frac{\Delta f}{\Delta x}(a)\ \doteq\ 0.
  \]
\end{proof}

\begin{proof}[A proof of Proposition \ref{821}.]
  Put $\alpha_i(x)=\left[\frac{x}{\varepsilon_i}\right]\cdot\varepsilon_i$ for $x\in[0,1]_Q\equiv\left\{ z\in Q\, ;\, 0\leq z\leq 1\right\}$ and $i\in\{1,2\}$.
  Then, for any $a\in[0,1]_Q$ the next property holds by Theorem \ref{831}:
  \[
  \left(
  \begin{matrix}
  \left(\exists\eta\in\mon{}{0}\right)
  \left(y-\alpha_i(a)>|\eta|\right)
  \wedge
  \left(y\doteq\alpha_i(a)\right)&&\\
  &&\hspace{-4.32cm}\Rightarrow
  \left(\frac{f_i(y)-f_i(\alpha_i(a))}{y-\alpha_i(a)}-\frac{\Delta f_i}{\Delta x}(\alpha_i(a))\doteq 0\right)
  \end{matrix}
  \right).
  \]
  By Proposition \ref{753}, the following is met:
  \[
  \left(
  \begin{matrix}
  \left(\exists\delta\in\mon{}{0}\right)
  \left(x-a>|\delta|\right)
  \wedge
  \left( x\doteq a\right)&&\\
  &&\hspace{-4.32cm}\Rightarrow
  \left(\frac{f_i(\alpha_i(x))-f_i(\alpha_i(a))}{x-a}-\frac{\Delta f_i}{\Delta x}(\alpha_i(a))\doteq 0\right)
  \end{matrix}
  \right).
  \]
  Since $f_1\circ\alpha_1\doteq f_2\circ\alpha_2$, Proposition \ref{751} implies
  \[
  \left(
  \begin{matrix}
  \left(\exists\delta\in\mon{}{0}\right)
  \left(x-a>|\delta|\right)
  \wedge
  \left( x\doteq a\right)&&\\
  &&\hspace{-4.32cm}\Rightarrow
  \left(\frac{\Delta f_1}{\Delta x}(\alpha_1(a)) - \frac{\Delta f_2}{\Delta x}(\alpha_2(a))
    \doteq 0\right)
  \end{matrix}
  \right).
  \]
  Hence, $\frac{\Delta f_1}{\Delta x}(\alpha_1(a)) \doteq \frac{\Delta f_2}{\Delta x}(\alpha_2(a))$ is fulfilled.
\end{proof}

Now, a real function $\mathscr{F}$ on a real interval $[0,1]$ is said to be \textit{differentiable} if it is represented by $\left(f,[0,1]_\varepsilon\right)$ and its difference quotient function $\frac{\Delta f}{\Delta x}$ is continuous.
The real function represented by $\left(\frac{\Delta f}{\Delta x}, [0,1]_{\varepsilon}\right)$ is a \textit{derivative} of $\mathscr{F}$ denoted by $\mathscr{F}'$.

It may seem unclear whether it is compatible with the derivative derived by an \textit{ordinary manner}.
As a matter of fact, it is.
Specifically, its representation is indiscernible with the difference quotient function $\frac{\Delta f}{\Delta x}$.
To confirm that, let us first introduce notions of real sequences and their convergence.

Let $\left(\mon{\mathit{}}{a_n}\right)_{n\in \tau}$ and $(a_n)_{n\in \tau}$ denote sequences of real numbers and their positions, which consist of $\tau\in N\setminus\FN$ bounded rational numbers, respectively.
A real sequence $(\mon{}{a_i})_{i\in\FN}$  is said to \textit{converge\footnote{
  See Sakahara and Sato \cite{topology} for the statement under continua in general.} to} $\mon{}{x}$, or equivalently that of rationals $(a_i)_{i\in\FN}$ to $x$, iff
\[
\left(
\forall q\in\FQ
\right)
\left(
\exists i\in\FN
\right)
\left(
\forall j>i
\right)
\left(
| a_j - x|< |q|
\right).
\]
It is simply denoted as $\lim_{i\in\FN} \mon{}{a_i}\ = \ \mon{}{x}$, or equivalently $\lim_{i\in\FN} a_i\doteq x$, iff $(\mon{}{a_i})_{i\in\FN}$ converges to $\mon{}{x}$.
Let us simply denote $\lim_{t\rightarrow 0} t\doteq 0$ iff there exists a rational sequence $(a_i)_{i\in\FN}$ which converges to 0, or $\lim_{i\in\FN} a_i\doteq 0$.

\begin{thm}\label{6}
  Let $\mathscr{F}$ be a differentiable real function and $f$ be its representation.
  Then, the following indiscerinibility is met
  \[
  \lim_{t\rightarrow 0}\frac{f(x+t)-f(x)}{t}\ \doteq\ \frac{\Delta f}{\Delta x}(x).
  \]
\end{thm}
\begin{proof}
  Let $(a_n)_{n\in\tau}$, in which $\tau\in N\setminus\FN$, be a decreasing sequence on $[0,1]_\varepsilon$ such that $a_n\not\doteq 0$ for all $n\in\FN$ and $a_n\doteq 0$ but $a_n\ne 0$ otherwise.
  Choose $x\in[0,1]_\varepsilon$ arbitrarily.
  Suppose there exists $q\in\FQ$ which satisfies for all $i\in\FN$  that there always exists $j>i$ fulfilling the inequality
  \[
  \left|
  \frac{f(x+a_j)-f(x)}{a_j} - \frac{\Delta f}{\Delta x}(x)
  \right|
  \  >\ |q|.
  \]
  Then, by the axiom of prolongation, there exists $\alpha\in\tau\setminus\FN$ which satisfies
  \[
  \left|
  \frac{f(x+a_\alpha)-f(x)}{a_\alpha} - \frac{\Delta f}{\Delta x}(x)
  \right|
  \  >\ |q|.
  \]
  But by Theorem \ref{831}, if $a_\alpha>|\eta|>0$ is satisfied for some $\eta\in\mon{}{0}$, the following indiscernibility equivalence must be satisfied
  \[
  \frac{f(x+a_\alpha)-f(x)}{a_\alpha}\ =\
  \frac{f\upharpoonright [0,1]_\varepsilon(x+a_\alpha)-f\upharpoonright [0,1]_\varepsilon(x)}{a_\alpha}
  \ \doteq\ \frac{\Delta f}{\Delta x}(x).
  \]
  It is a contradiction.
\end{proof}

The basic properties of derivatives such as chain rule or inverse function theorem are also preserved (see 8.4 and 8.5 of Tsujishita \cite{tjst}, respectively).
Higher order derivative and its differentiability are drawn in a similar way (see 8.6 and 8.7 of Tsujishita \cite{tjst}).

The rational-valued function $\Sigma f\Delta x$ on $[0,1]_\varepsilon$, which is continuous and finite, is given as
\[
\left(\Sigma f\Delta x\right)(u)\ \equiv\ \sum_{0\leq x\leq u} f(x)\Delta x.
\]
An \textit{indefinite integral} of $\mathscr{F}$, denoted as $\int_0^t \mathscr{F}(x)dx$, is the real function represented by $\Sigma f\Delta x$.
Its continuity and finiteness are verified by Proposition 8.8.1 and its independence from the choice of the representation by Proposition 8.8.3 of Tsujishita \cite{tjst}.

Then the fundamental theorem of calculus is established.
\begin{prop}[Proposition 8.8.4 of Tsujishita \cite{tjst}]
  Suppose $\mathscr{F}$ is a real function on $[0,1]$.
  Then the real function $\int_0^t \mathscr{F}(x)dx$ on $[0,1]$ is differentiable and its derivative is $\mathscr{F}$.
\end{prop}

\bibliographystyle{plain}
\bibliography{ref}

\end{document}